\documentclass[intlimits,tbtags]{article}
\usepackage[utf8]{inputenc}
\usepackage{amstext}

\makeatletter

\usepackage[T1]{fontenc}    

\usepackage{a4wide}        
\usepackage{amsmath}

\usepackage{euscript}
\usepackage{mathtools}
\allowdisplaybreaks[1] 
\usepackage{amsfonts}

\usepackage{mathtools}
\usepackage{amsthm}

\newcommand*{\id}{{\normalfont\hbox{1\kern-0.4em1}}}
\newcommand{\hide}[1]{}

\newcommand*{\curl}{\operatorname{curl}}
\newcommand*{\grad}{\operatorname{grad}}

\newcommand*{\dive}{\operatorname{div}}
\newcommand*{\Grad}{\operatorname{Grad}}

\newcommand*{\Div}{\operatorname{Div}}

\DeclareMathAccent{\Circ}{\mathalpha}{operators}{"17}
\newcommand{\interior}[1]{\Circ{#1}}

\renewcommand{\Re}{\operatorname{\mathfrak{Re}}}

\newcommand{\oi}[2]{\left]#1,#2 \right[}

\renewcommand{\tilde}{\widetilde}
\renewcommand*{\epsilon}{\varepsilon}
\renewcommand*{\rho}{\varrho}
\theoremstyle{plain}
\newtheorem{thm}{Theorem}[section]
  \theoremstyle{definition}
  \newtheorem{defn}[thm]{Definition}
  \theoremstyle{remark}
  \newtheorem{rem}[thm]{Remark}
 \theoremstyle{definition}
  
  \theoremstyle{plain}
  
  \theoremstyle{plain}
  
  \theoremstyle{plain}
  \newtheorem*{lem*}{Lemma}
  \theoremstyle{plain}
  \newtheorem*{thm*}{Theorem}
  \theoremstyle{plain}
  
  \theoremstyle{remark}

\makeatother

\begin{document}
    
\title{On Well-Posedness for a Piezo-Electromagnetic Coupling Model with Boundary Dynamics.}
\author{Rainer Picard }  
   

\maketitle 
\abstract{We consider a coupled system of Maxwell's equations and the equations of elasticity, which is commonly used to model  piezo-electric material behavior. The boundary influence is encoded as a separate  dynamics  on the boundary data spaces coupled to the partial differential equations. Evolutionary well-posedness, i.e. Hadamard well-posedness and causal dependence on the data, is shown for the resulting model system.}   

\section{Introduction.}

There is an abundance of applications for piezoelectric materials.
Their primary use is in ultrasonic transducers. Typical applications
can be found in medical imaging and non-destructive testing of safety
critical structures. The well-posedness of corresponding models, which
is the focus of this paper, is of interest in the evaluation of respective
models and in particular as a basis for inverse problems. A useful
summary of the literature that has examined well-posedness issues
for a range of boundary conditions can be found in \cite{akamatsu2002}
. 

In this paper we will consider a model class discussed in \cite{zbMATH05929394}
and generalize it to a broader class of problems, where for example
the operator coefficients could be non-local, e.g. of convolution
type, and will not be restricted to just multiplication operators.
To be concrete a coefficient operator $\alpha$ may for example be
given in the form
\[
\left(\alpha f\right)\left(x\right)\coloneqq\alpha_{0}\left(x\right)f\left(x\right)+\int_{\Omega}\alpha_{1}\left(x,y\right)f\left(y\right)dy,
\]
where $\Omega$ is the underlying spatial domain carrying the material
properties described by $\alpha$. Another common way non-locality
of coefficients can come into play is via orthogonal projectors entering
the coefficient operators, e.g. Helmholtz projectors. 

More importantly, there will be \emph{no} constraints on the boundary
quality of $\Omega$ so that more complex configurations such as materials
with fractal boundaries, which have been considered and even prototyped
more recently, see e.g. \cite{walker2011}, come into reach. We shall
propose a generalized boundary condition, which in fact takes on the
form of an extra equation describing the dynamics on the topological
boundary set $\dot{\Omega}$ of the underlying non-empty open set
$\Omega$. For computational purposes one would have to assume approximations
by domains with better boundary quality such as a Lipschitz boundary
in which case classical boundary trace operators can be utilized.
To pave the way a discussion of classical boundary trace arguments
and our abstract characterization of boundary data spaces is also
included.

Since in the general situation we consider here boundary trace theorems
are not available, the analysis is based on an alternative characterization
of boundary data, which makes no reference to the boundary quality.

Our discussion will be based on the space-time Hilbert space framework
developed in \cite{Pi2009-1}, see also \cite{PDE_DeGruyter}, for
what we shall call \emph{evo-systems}. After briefly recalling the
conceptual building blocks of the theoretical framework in Section
\ref{sec:A-Brief-Summary} we then establish the classical system
of piezo-electro-magnetism with standard homogeneous boundary conditions
as such a system (Section \ref{sec:The-Evo-System-of}). In Section
\ref{sec:Inhomogeneous-Boundary-Condition} we initially discuss standard
inhomogeneous boundary conditions to introduce the boundary data characterization
utilized in our general setting, in particular Subsection \ref{subsec:Inhomogeneous-Initial-Boundary}.
Then the more complex situation of a Leontovich type boundary condition
is explored within this boundary data space framework in Subsections
\ref{subsec:Leontovich-Type-Boundary}. Rather than implementing this
type of boundary constraint into the differential operator domain,
as is standard for the classical Dirichlet and Neumann type boundary
condition, this mixed type boundary condition is described via additional
dynamic equations in the boundary data spaces.

\section{\label{sec:A-Brief-Summary}A Brief Summary of Evo-Systems}

The solution theory of the class of so-called \emph{evolutionary}
equations (evo-systems for short) introduced in \cite{Pi2009-1} is
based on the fact that the (time) derivative $\partial_{0}$ is, in
a suitable setting, a normal operator with a strictly positive real
part. Indeed, in the space $H_{\nu,0}\left(\mathbb{R},H\right)$,
$\nu\in\oi0\infty$ , ~ of $H$-valued $L^{2,\mathrm{loc}}$-functions
($H$ a Hilbert space with inner product $\left\langle \:\cdot\:|\:\cdot\:\right\rangle _{H}$)
equipped with the inner product $\left\langle \:\cdot\:|\:\cdot\:\right\rangle _{\nu,0,H}$
\[
\left(\varphi,\psi\right)\mapsto\int_{\mathbb{R}}\left\langle \varphi\left(t\right)|\psi\left(t\right)\right\rangle _{H}\;\exp\left(-2\nu t\right)\:dt,
\]
we have that $\partial_{0}$ is a normal operator, i.e. commuting
with its adjoint on $D\left(\partial_{0}^{2}\right)$, and 
\[
\Re\partial_{0}=\nu>0.
\]
Throughout, we denote by $\partial_{0}$ this derivative as a derivative
with respect to time. Under suitable assumptions the latter property
of $\partial_{0}$ can be carried over to problems of the general
form

\begin{equation}
\overline{\partial_{0}M\left(\partial_{0}^{-1}\right)+A}\:U=F,\label{eq:evo}
\end{equation}
where now $A:D\left(A\right)\subseteq H\to H$ is a closed densely
defined linear operator and $\left(M\left(z\right)\right)_{z\in B_{\mathbb{C}}\left(r,r\right)}$
($B_{\mathbb{C}}\left(r,r\right)$ denotes the open ball in $\mathbb{C}$
of radius $r\in\oi0\infty$ centered at $r\in\oi0\infty$~) is a
uniformly bounded analytic operator family. The well-posedness of
\eqref{eq:evo} can be based on strict (real) positive definiteness
of $\left(\partial_{0}M\left(\partial_{0}^{-1}\right)+A\right)$ and
its adjoint for all sufficiently large weight parameters $\nu\in\oi0\infty$~.
The resulting problem class is referred to as \emph{evolutionary equations}
to contrast it with classical evolution equations in Hilbert space,
which are a special case. For emphasis we shall use the term ``evo-system''
for problems of this class, since classical evolution equations are
sometimes also referred to as ``evolutionary''. 

In this paper we shall be dealing with a rather special and so also
more easily accessible case. We only need to consider the case, where
$A$ is skew-selfadjoint and $z\mapsto M\left(z\right)$ is actually
a rational (operator-valued) function defined in a neighborhood of
the origin. 

To recall the solution theory (as described in the last chapter of
\cite{PDE_DeGruyter}) for our somewhat simpler situation the needed
requirement is that $M\left(0\right)$ is selfadjoint and that
\begin{equation}
\nu M\left(0\right)+\Re M^{\prime}\left(0\right)\geq c_{0}>0\mbox{ for all sufficiently large }\nu\in]0,\infty[.\:\label{eq:posdef}
\end{equation}
Equation (\ref{eq:posdef}) is for example satisfied if $M\left(0\right)$
is strictly positive definite on its range and $\Re M^{\prime}\left(0\right)$
strictly positive definite on the null space of $M\left(0\right)$,
which will turn out to be valid in our present application.

\begin{rem}Whenever we are not interested in the actual constant
$c_{0}\in\oi0\infty$ we shall write for the strict positive definiteness
constraint
\[
\Re T\geq c_{0}
\]
simply
\[
T\gg0.
\]
If we want to state that there is such a constant $c_{0}$ for a whole
family of operators $\left(T_{\nu}\right)_{\nu\in I}$, we say that
\[
T_{\nu}\gg0
\]
holds uniformly with respect to $\nu$. 

So, the general requirement for the problem class under consideration
would be stated as $M\left(0\right)$ selfadjoint and
\begin{equation}
\nu M\left(0\right)+\Re M^{\prime}\left(0\right)\gg0\label{eq:pos-def11}
\end{equation}
uniformly for all sufficiently large $\nu\in]0,\infty[$~.\end{rem}

\begin{defn} A problem class is called Hadamard well-posed, if we
have existence, uniqueness of a solution as well as continuous dependence
of the solution on the data. For dynamic problems we also want causal
dependence on the data. We shall call the problem class described
by an evo-system as well-posed, if there exists a continuous linear
solution operator $\mathcal{S}$ (Hadamard-wellposedness), which moreover
satisfies the causality condition\footnote{Here $\chi_{_{I}}\left(m_{0}\right)$ denotes the temporal cut-off
by the characteristic function of $I$
\[
\left(\chi_{_{I}}\left(m_{0}\right)f\right)\left(t\right)=\begin{cases}
f\left(t\right) & \text{ for }t\in I\text{,}\\
0 & \text{ otherwise.}
\end{cases}
\]
} 
\[
\chi_{_{]-\infty,a[}}\left(m_{0}\right)\;\mathcal{S}\:\chi_{_{[a,\infty[}}\left(m_{0}\right)=0
\]
for all $a\in\mathbb{R}$ (causality).\end{defn}

For sake of reference we record the corresponding well-posedness result
for evo-systems.

\begin{thm}\label{thm:well-posed}Let $A:D\left(A\right)\subseteq H\to H$
be skew-selfadjoint and $z\mapsto M\left(z\right)$ be a uniformly
bounded, linear-operator-valued rational function in a neighborhood
of zero such that \eqref{eq:pos-def11} is satisfied uniformly for
all $\nu\in[\nu_{0},\infty[$~, for some $\nu_{0}\in\oi0\infty$~.
Then the evo-system \eqref{eq:evo} is well-posed. \end{thm}

Thus, we have not only that for every $F\in H_{\nu,0}\left(\mathbb{R},H\right)$
there is a unique solution $U\in D\left(\overline{\partial_{0}M\left(\partial_{0}^{-1}\right)+A}\right)$,
but also that the solution operator $\overline{\partial_{0}M\left(\partial_{0}^{-1}\right)+A}^{-1}:H_{\nu,0}\left(\mathbb{R},H\right)\to H_{\nu,0}\left(\mathbb{R},H\right)$
is a continuous linear mapping, which, moreover, is also causal in
the sense that
\[
\chi_{_{]-\infty,a[}}\left(m_{0}\right)\;\overline{\partial_{0}M\left(\partial_{0}^{-1}\right)+A}^{-1}\:\chi_{_{[a,\infty[}}\left(m_{0}\right)=0
\]
for all $a\in\mathbb{R}$.

On occasion, we also want to use some additional regularity observations,
which we therefore also introduce here. For this we need some dual
spaces. We choose to identify
\begin{eqnarray*}
H & = & H^{\prime}
\end{eqnarray*}
and 
\begin{eqnarray*}
H_{\nu,0}\left(\mathbb{R},H\right) & = & \left(H_{\nu,0}\left(\mathbb{R},H\right)\right)^{\prime},
\end{eqnarray*}
and we define $H_{\nu,1}\left(\mathbb{R},H\right)$ as the domain
of $\partial_{0}$ equipped with the norm induced by the inner product
$\left\langle \:\cdot\:|\:\cdot\:\right\rangle _{\nu,1,H}\coloneqq\left\langle \partial_{0}\:\cdot\:|\partial_{0}\:\cdot\:\right\rangle _{\nu,0,H}$
as well as

\begin{eqnarray*}
H_{\nu,-1}\left(\mathbb{R},H\right) & := & \left(H_{\nu,1}\left(\mathbb{R},H\right)\right)^{\prime}.
\end{eqnarray*}
We also shall make use of the Hilbert space 
\[
H_{\nu,-1}\left(\mathbb{R},D\left(A^{*}\right)^{\prime}\right):=H_{\nu,1}\left(\mathbb{R},D\left(A^{*}\right)\right)^{\prime},
\]
where we canonically consider $D\left(C\right)$ with a closed operator
$C$ as a Hilbert space with respect to the graph inner product. Denoting
again by $A$ the continuous extension of
\begin{align*}
D\left(A\right)\subseteq H & \to D\left(A^{*}\right)^{\prime}\\
x & \mapsto Ax
\end{align*}
we have with this, letting $M_{0}:=M\left(0\right)$, $M_{1}\left(\partial_{0}^{-1}\right):=\partial_{0}\left(M\left(\partial_{0}^{-1}\right)-M\left(0\right)\right)$,
that 
\[
\partial_{0}M_{0}U=F-M_{1}\left(\partial_{0}^{-1}\right)U-AU\in H_{\nu,0}\left(\mathbb{R},D\left(A^{*}\right)^{\prime}\right)
\]
and so we read off that
\begin{equation}
M_{0}U\in H_{\nu,1}\left(\mathbb{R},D\left(A^{*}\right)^{\prime}\right).\label{eq:time-reg}
\end{equation}
We similarly have
\[
AU=F-M_{1}\left(\partial_{0}^{-1}\right)U-M_{0}\partial_{0}U\in H_{\nu,-1}\left(\mathbb{R},H\right)
\]
and so

\begin{equation}
U\in H_{\nu,-1}\left(\mathbb{R},D\left(A\right)\right).\label{eq:space-reg}
\end{equation}

Note that for the solution $U$ according to Theorem \ref{thm:well-posed}
we not only have the regularity statements \eqref{eq:time-reg}, \eqref{eq:space-reg},
but also that the equation 
\[
\partial_{0}M_{0}U+M_{1}\left(\partial_{0}^{-1}\right)U+AU=F
\]
holds in $H_{\nu,-1}\left(\mathbb{R},D\left(A^{*}\right)^{\prime}\right).$
We shall use the latter fact as motivation to drop henceforth the
closure bar for equations of the form \eqref{eq:evo}. 

One of the foremost complications in practical applications is that
the evo-system structure is frequently obscured. This is mostly the
case due to purely formal, i.e. informal, calculations performed under
unclear assumptions in the modeling process. To address rigorous ways
to produce equations equivalent to evo-systems we recall the following
linear algebra terminology.

\begin{defn}If continuous, linear Hilbert space bijections $\mathcal{W},\mathcal{V}$
exist such that 
\[
\mathcal{B}=\mathcal{W}\mathcal{A}\mathcal{V},
\]
then $\mathcal{A}$ and $\mathcal{B}$ are called equivalent. If $\mathcal{V}=\mathcal{W}^{*}$
then $\mathcal{A}$ and $\mathcal{B}$ are called congruent. If $\mathcal{V}=\mathcal{W}^{-1}$
then $\mathcal{A}$ and $\mathcal{B}$ are called similar. If $\mathcal{V},\mathcal{W}$
are unitary then $\mathcal{A}$ and $\mathcal{B}$ are called unitarily
equivalent, unitarily congruent (or unitarily similar), respectively.\end{defn}

\begin{rem}\label{rem:If-continuous,-linear}Equivalence in the stated
sense preserves Hadamard well-posedness.\footnote{This fact is actually the reason for the choice of the term ``equivalence''.}
For an equivalent equation it may, however, be hard to detect further
structural properties, since for example (skew-)selfadjointness gets
easily lost in the process. 

In contrast, \emph{spatial congruence}, i.e. \emph{$\mathcal{W}$}
only acts on the spatial Hilbert space $H$, is, if lifted to the
time-dependent case, a structure preserving operation for evo-systems.
Indeed, for $\mathcal{W}:H\to X$
\begin{eqnarray*}
\mathcal{W}F & = & \mathcal{W}\left(\partial_{0}M_{0}+M_{1}\left(\partial_{0}^{-1}\right)+A\right)\mathcal{W}^{*}\left(\left(\mathcal{W}^{-1}\right)^{*}U\right)\\
 & = & \left(\partial_{0}\mathcal{W}M_{0}\mathcal{W}^{*}+\mathcal{W}M_{1}\left(\partial_{0}^{-1}\right)\mathcal{W}^{*}+\mathcal{W}A\mathcal{W}^{*}\right)\left(\left(\mathcal{W}^{-1}\right)^{*}U\right)
\end{eqnarray*}
where now $\mathcal{W}A\mathcal{W}^{*}$ is still skew-selfadjoint
and $\mathcal{W}M_{0}\mathcal{W}^{*}$ is still selfadjoint. Assuming
that \eqref{eq:posdef} holds, we find
\begin{eqnarray*}
\left\langle U|\nu\mathcal{W}M_{0}\mathcal{W}^{*}U+\Re\left(\mathcal{W}M_{1}\left(0\right)\mathcal{W}^{*}\right)U\right\rangle _{X} & = & \left\langle \mathcal{W}^{*}U|\left(\nu M_{0}+\Re\left(M_{1}\left(0\right)\right)\right)\mathcal{W}^{*}U\right\rangle _{H}\\
 & \geq & c_{0}\left\langle \mathcal{W}^{*}U|\mathcal{W}^{*}U\right\rangle _{H}\\
 & \geq & c_{0}\left\Vert \left(\mathcal{W}^{-1}\right)^{*}\right\Vert ^{-2}\left\langle U|U\right\rangle _{X}
\end{eqnarray*}
where we have used
\begin{eqnarray*}
\left\langle U|U\right\rangle _{X} & = & \left\langle \left(\mathcal{W}^{-1}\right)^{*}\mathcal{W}^{*}U|\left(\mathcal{W}^{-1}\right)^{*}\mathcal{W}^{*}U\right\rangle _{X}\\
 & \leq & \left\Vert \left(\mathcal{W}^{-1}\right)^{*}\right\Vert ^{2}\left\langle \mathcal{W}^{*}U|\mathcal{W}^{*}U\right\rangle _{H}.
\end{eqnarray*}
In particular, \eqref{eq:pos-def11} remains satisfied, i.e. $\left(\partial_{0}\mathcal{W}M_{0}\mathcal{W}^{*}+\mathcal{W}M_{1}\left(\partial_{0}^{-1}\right)\mathcal{W}^{*}+\mathcal{W}A\mathcal{W}^{*}\right)$
is an evo-system operator in $H_{\nu,0}\left(\mathbb{R},X\right)$,
where we originally had an evo-system in $H_{\nu,0}\left(\mathbb{R},H\right)$.\end{rem}

\section{\label{sec:The-Evo-System-of}The Evo-System of Piezo-Electro-Magnetism}

\subsection{\label{subsec:The-Basic-System}The Basic System}

Let $\Omega\subseteq\mathbb{R}^{3}$ be an arbitrary non-empty open
set. The system of Piezo-Electro-Magnetism in a medium occupying $\Omega$
is a coupled system consisting of the equation of elasticity and Maxwell's
equations. The equation of elasticity is given by
\begin{align}
\partial_{0}^{2}\rho_{\ast}u-\Div T & =F_{0},\label{eq:elasticity}
\end{align}
where $u:\mathbb{R}\times\Omega\to\mathbb{R}^{3}$ describes the displacement
of the elastic body $\Omega,$ $T:\mathbb{R}\times\Omega\to\mathrm{sym}\left[\mathbb{R}^{3\times3}\right]$
denotes the stress tensor, which is assumed to attain values in the
space of symmetric matrices. The function $\rho_{\ast}:\Omega\to\mathbb{R}$
stands for the density of $\Omega$ and $F_{0}:\mathbb{R}\times\Omega\to\mathbb{R}^{3}$
is an external force term. Maxwell's equations are given by
\begin{align}
\partial_{0}B+\curl E & =F_{1},\nonumber \\
\partial_{0}D-\curl H & =-J_{0}-\sigma E.\label{eq:Maxwell}
\end{align}
Here, $B,D,E,H:\mathbb{R}\times\Omega\to\mathbb{R}^{3}$ denote the
magnetic flux density, the electric displacement, the electric field
and the magnetic field, respectively. The functions $J_{0},F_{1}:\mathbb{R}\times\Omega\to\mathbb{R}^{3}$
are given source terms and $\sigma:\Omega\to\mathbb{R}^{3\times3}$
denotes the resistance tensor. Of course, all these equations need
to be completed by suitably modified material laws, where also the
coupling will occur. As it will turn out, the system can be written
in the following abstract form
\begin{equation}
\begin{array}{r}
\left(\partial_{0}M_{0}+M_{1}\left(\partial_{0}^{-1}\right)+\left(\begin{array}{cccc}
0 & -\Div & 0 & 0\\
-\Grad & 0 & 0 & 0\\
0 & 0 & 0 & -\curl\\
0 & 0 & \curl & 0
\end{array}\right)\right)\left(\begin{array}{c}
v\\
T\\
E\\
H
\end{array}\right)=\\
=\left(\begin{array}{c}
F_{0}\\
G_{0}\\
-J_{0}\\
F_{1}
\end{array}\right),
\end{array}\label{eq:system}
\end{equation}
for a suitable bounded operator $M_{0}$ and a uniformly bounded rational
operator family $\left(M_{1}\left(z\right)\right)_{z\in U}$, $U$
a neighborhood of zero, on the Hilbert space $H:=L^{2}(\Omega)^{3}\oplus\mathrm{sym}\left[L^{2}(\Omega)^{3\times3}\right]\oplus L^{2}(\Omega)^{3}\oplus L^{2}(\Omega)^{3}$.
Here, $v:=\partial_{0}u$. 

Of course, we also need to impose boundary constraints. To make this
precise, we need to properly define the spatial differential operators.

\begin{defn}We denote by $\interior C_{\infty}(\Omega)$ the space
of arbitrarily differentiable functions with compact support in $\Omega.$
Then we define the operator $\interior\curl$ as the closure of
\begin{align*}
\interior C_{\infty}(\Omega)^{3}\subseteq L^{2}(\Omega)^{3} & \to L^{2}(\Omega)^{3}\\
(\phi_{1},\phi_{2},\phi_{3}) & \mapsto\left(\begin{array}{ccc}
0 & -\partial_{3} & \partial_{2}\\
\partial_{3} & 0 & -\partial_{1}\\
-\partial_{2} & \partial_{1} & 0
\end{array}\right)\left(\begin{array}{c}
\phi_{1}\\
\phi_{2}\\
\phi_{3}
\end{array}\right)
\end{align*}
and $\curl:=\left(\interior\curl\right)^{\ast}\supseteq\interior\curl.$
We also define $\interior\Grad$ and $\interior\Div$ as the closures
of
\begin{align*}
\interior C_{\infty}(\Omega)^{3}\subseteq L^{2}(\Omega)^{3} & \to\mathrm{sym}\left[L^{2}(\Omega)^{3\times3}\right]\\
(\phi_{1},\phi_{2},\phi_{3}) & \mapsto\frac{1}{2}\left(\partial_{j}\phi_{i}+\partial_{i}\phi_{j}\right)_{i,j\in\{1,2,3\}}
\end{align*}
and of
\begin{align*}
\mathrm{sym}\left[\interior C_{\infty}(\Omega)^{3\times3}\right]\subseteq\mathrm{sym}\left[L^{2}(\Omega)^{3\times3}\right] & \to L^{2}(\Omega)^{3}\\
(\phi_{ij})_{i,j\in\{1,2,3\}} & \mapsto\left(\sum_{j=1}^{3}\partial_{j}\phi_{ij}\right)_{i\in\{1,2,3\}},
\end{align*}
respectively, and set $\Grad:=-\left(\interior\Div\right)^{\ast}$
as well as $\Div:=-\left(\interior\Grad\right)^{\ast}.$ Here $L^{2}(\Omega)^{3\times3}$
has a Hilbert space structure unitarily equivalent to $L^{2}(\Omega)^{9}$.
Elements in the domain of the operators marked by a overset circle
satisfy an abstract homogeneous boundary condition, which, in case
of a sufficiently smooth boundary $\partial\Omega$ (e.g. a Lipschitz
boundary), can be written as
\[
u=0\mbox{ on }\partial\Omega
\]
for $u\in D(\interior\Grad)$,
\[
Tn=0\mbox{ on }\partial\Omega
\]
for $T\in D(\interior\Div),$ where $n$ denotes the exterior unit
normal vector field on $\partial\Omega$ and
\[
E\times n=0\mbox{ on }\partial\Omega,
\]
for $E\in D(\interior\curl).$\end{defn}

Not to incur unnecessary constraints on the boundary quality we shall,
however, use the generalized homogeneous boundary conditions of containment
in $D(\interior\Grad),\:D(\interior\Div),\:D(\interior\curl)$, respectively. 

For sake of definiteness we shall focus for now on the classical Dirichlet
case: $n\times E=0$, $v=0$ on the boundary $\partial\Omega$, i.e.
in generalized terms on the system
\begin{equation}
\begin{array}{r}
\left(\partial_{0}M_{0}+M_{1}\left(\partial_{0}^{-1}\right)+\left(\begin{array}{cccc}
0 & -\Div & 0 & 0\\
-\interior\Grad & 0 & 0 & 0\\
0 & 0 & 0 & -\curl\\
0 & 0 & \interior\curl & 0
\end{array}\right)\right)\left(\begin{array}{c}
v\\
T\\
E\\
H
\end{array}\right)=\\
=\left(\begin{array}{c}
F_{0}\\
G_{0}\\
-J_{0}\\
F_{1}
\end{array}\right).
\end{array}\label{eq:Dirichlet-evo}
\end{equation}
We still need to specify the material law of interest. 

\subsection{\label{subsec:The-Material-Relations}The Material Relations of Piezo-Electro-Magnetism}

In this section we discuss material relations suggested in \cite{mindlin1974}
and derive the structure of the corresponding operators $M_{0}$ and
$M_{1}$. Furthermore, we give sufficient conditions on the parameters
involved to warrant the solvability condition (\ref{eq:pos-def11}).\\
The material relations described in \cite{mindlin1974} are initially
given (ignoring for simplicity thermal couplings) in the form (where
we write $\mathcal{E}=\Grad u$ as usual for the strain tensor)
\begin{eqnarray*}
T & = & C\:\mathcal{E}-eE,\\
D & = & e^{*}\mathcal{E}+\epsilon E,\\
B & = & \mu\,H.
\end{eqnarray*}
Here $C\in L\left(\mathrm{sym}\left[L^{2}(\Omega)^{3\times3}\right]\right)$
is the (invertible) elasticity ``tensor''\footnote{Since every linear mapping $F:X\to Y$ can be interpreted as a bilinear
functional $\left(\left(x,y\right)\mapsto y\left(Fx\right)\right)\in\left(X\otimes Y^{\prime}\right)^{\prime}$
the term tensor for $C$ is not completely misplaced. It supports,
however, a common misunderstanding that $C$ is considered to be a
tensor\emph{ field}, where it is indeed just a \emph{mapping} between
symmetric tensor field. The mapping $C$ can only be considered as
a tensor field if we would restrict our attention to multiplicative
mappings, i.e. 
\[
\left(C\mathcal{E}\right)\left(x\right)\coloneqq\tilde{C}\left(x\right)\mathcal{E}\left(x\right)\:a.e.
\]
for an $L^{\infty}$-function $\tilde{C}$ from $\Omega$ to $L\left(\mathrm{sym}\left[\mathbb{R}^{3\times3}\right]\right)$,
which expressly we do not want to limit ourselves to, then $C$ itself
could also be \emph{interpreted} as a tensor field $\left(C_{ij}^{\:\;\;kl}\left(x\right)\right)_{i,j,k,l}$
so that
\[
\left(C\mathcal{E}\right)\left(x\right)=\left(C_{ij}^{\:\;\;kl}\left(x\right)\mathcal{E}_{kl}\left(x\right)\right)_{i,j}.
\]
Since $C$ is supposed to map symmetric tensor fields to symmetric
tensor fields we must have – in this case – the well-known symmetry
relations for the real-valued functions ($g$ denotes the metric tensor)
\[
C^{ijkl}\left(x\right)\coloneqq\sum_{s,t=1}^{3}g^{is}\left(x\right)g^{jt}\left(x\right)C_{st}^{\:\;\;kl}\left(x\right)\:a.e.\,,\;i,j,k,l=1,2,3,
\]
namely that
\[
C^{ijkl}\left(x\right)=C^{ijlk}\left(x\right)=C^{jikl}\left(x\right)\:a.e.\,,\;i,j,k,l=1,2,3.
\]
The also assumed selfadjointness of $C$ clearly results in another
set of symmetry relations
\[
C^{ijkl}\left(x\right)=C^{klij}\left(x\right)\:a.e.\,,\;i,j,k,l=1,2,3,
\]
which is like-wise a standard requirement in this context.}, $\varepsilon,\mu\in L\left(L^{2}(\Omega)^{3}\right)$ are the permittivity
and permeability, respectively, all assumed to be selfadjoint and
non-negative. The notation $L\left(X_{0},X_{1}\right)$ is used to
denote the Banach space of continuous linear mappings from the Hilbert
space $X_{0}$ to the Hilbert space $X_{1}$. In the case $X_{0}=X_{1}$
we write, as done here, more concisely $L\left(X_{0}\right)$. The
operator 
\[
e\in L\left(L^{2}(\Omega)^{3},\mathrm{sym}\left[L^{2}(\Omega)^{3\times3}\right]\right)
\]
acts as a coupling ``parameter''. To adapt the material relations
to our framework we solve for $\mathcal{E}$ and obtain
\begin{eqnarray*}
\mathcal{E} & = & C^{-1}T+C^{-1}eE,\\
D & = & e^{*}C^{-1}T+\left(\epsilon+e^{*}C^{-1}e\right)E,\\
B & = & \mu\:H.
\end{eqnarray*}
Thus, we arrive at a material law equation of the form 
\begin{eqnarray*}
\left(\begin{array}{c}
\rho_{*}v\\
\mathcal{E}\\
D\\
B
\end{array}\right) & = & M_{0}\left(\begin{array}{c}
v\\
T\\
E\\
H
\end{array}\right)+\partial_{0}^{-1}M_{1}\left(\begin{array}{c}
v\\
T\\
E\\
H
\end{array}\right).
\end{eqnarray*}
with 
\begin{equation}
\begin{array}{rcl}
M_{0} & = & \left(\begin{array}{cccc}
\rho_{*} & 0 & 0 & \quad0\\
0 & C^{-1} & C^{-1}e & \quad0\\
0 & e^{*}C^{-1} & \epsilon+e^{*}C^{-1}e & \quad0\\
0 & 0 & 0 & \quad\mu
\end{array}\right),\\
M_{1} & = & \left(\begin{array}{cccc}
0 & 0 & 0 & 0\\
0 & 0 & 0 & 0\\
0 & 0 & \sigma & 0\\
0 & 0 & 0 & 0
\end{array}\right).
\end{array}\label{eq:m01}
\end{equation}
Here $\sigma\in L\left(L^{2}(\Omega)^{3}\right)$ represents an additional
conductivity coefficient. \\
We need to ensure the solvability condition (\ref{eq:pos-def11})
with these material relations to obtain our first result.

\begin{thm}Assume that $\rho_{\ast},\varepsilon,\mu,C$ are selfadjoint
and non-negative. Furthermore, we assume $\rho_{\ast},\mu,C\gg0$
and $\nu\epsilon+\Re\sigma\gg0$ uniformly for all sufficiently large
$\nu\in\oi0\infty$~. Then, $M_{0}$ and $M_{1}$ given by \eqref{eq:m01}
satisfy the condition \eqref{eq:pos-def11} and hence, the corresponding
problem of piezo-electro-magnetism is a well-posed evo-system.\end{thm}

\begin{proof}Obviously, $M_{0}$ is selfadjoint. Moreover, since
$\rho_{\ast},\epsilon,\mu\gg0$, the only thing, which is left to
show, is that

\[
\nu\left(\begin{array}{cc}
C^{-1} & C^{-1}e\\
e^{\ast}C^{-1} & \varepsilon+e^{\ast}C^{-1}e
\end{array}\right)+\left(\begin{array}{cc}
0 & 0\\
0 & \Re\sigma
\end{array}\right)\gg0
\]
for all sufficiently large $\nu.$ By symmetric Gauss steps as congruence
transformations we get that the above operator is congruent to
\[
\nu\left(\begin{array}{cc}
C^{-1} & 0\\
0 & \varepsilon
\end{array}\right)+\left(\begin{array}{cc}
0 & 0\\
0 & \Re\sigma
\end{array}\right).
\]

The latter operator is then strictly positive definite by assumption
and so the assertion follows.\end{proof}

\section{\label{sec:Inhomogeneous-Boundary-Condition}Inhomogeneous Boundary
Conditions}

\subsection{Boundary Data Spaces}

Using that domains of closed, linear Hilbert space operators are themselves
in a canonical sense Hilbert spaces with respect to the associated
graph inner product we see that with
\begin{align*}
D(\interior\Grad) & \subseteq D(\Grad),\\
D(\interior\Div) & \subseteq D(\Div),\\
D(\interior\curl) & \subseteq D(\curl),
\end{align*}
we may consider the ortho-complements of the vanishing boundary data
spaces $D(\interior\Grad)$, $D(\interior\Div)$, $D(\interior\curl)$
in the larger Hilbert spaces $D(\Grad)$, $D(\Div)$, $D(\curl)$,
respectively. Prescribing boundary data for $D(\Grad)$, $D(\Div)$,
$D(\curl)$ can now be done conveniently by choosing elements of these
ortho-complements, which are
\[
N\left(1-\Div\Grad\right),\:N\left(1-\Grad\Div\right),\:N\left(1+\curl\curl\right),
\]
respectively. If $\iota_{\Grad}$, $\iota_{\Div}$, $\iota_{\curl}$
denote the canonical isometric, embeddings (i.e. via the identity)
of these null spaces into $D(\Grad)$, $D(\Div)$, $D(\curl)$, respectively,
then $\iota_{\Grad}^{*}$, $\iota_{\Div}^{*}$, $\iota_{\curl}^{*}$
perform the orthogonal projection\footnote{The more familiar corresponding orthogonal projectors from the projection
theorem context are
\[
P_{N\left(1-\Div\Grad\right)}=\iota_{\Grad}\iota_{\Grad}^{*},\;P_{N\left(1-\Grad\Div\right)}=\iota_{\Div}\iota_{\Div}^{*},\;P_{N\left(1+\curl\curl\right)}=\iota_{\curl}\iota_{\curl}^{*}.
\]
} onto the respective null spaces. With
\[
\overset{\bullet}{\Grad}:=\iota_{\Div}^{*}\Grad\iota_{\Grad},\;\overset{\bullet}{\Div}:=\iota_{\Grad}^{*}\Div\iota_{\Div},\:\overset{\bullet}{\curl}:=\iota_{\curl}^{*}\curl\iota_{\curl},
\]
we get that these are unitary mappings and 
\begin{eqnarray*}
\overset{\bullet}{\Grad}^{*} & = & \overset{\bullet}{\Div},\\
\overset{\bullet}{\curl}^{*} & = & -\overset{\bullet}{\curl}.
\end{eqnarray*}
Note that in contrast we have for example in $\mathbb{R}^{3}$
\[
\Grad^{*}=-\Div,\:\curl^{*}=\curl.
\]
This apparent contrast stems from the different choice of inner product
with respect to which the adjoints are constructed. To understand
this point let us recall from \cite{bath49528} the case of $\overset{\bullet}{\curl}:N\left(1+\curl\curl\right)\to N\left(1+\curl\curl\right)$
(the argument for $\overset{\bullet}{\Grad}$ being analogous). As
a closed subspace of the Hilbert space $D\left(\curl\right)$ the
inner product of $N\left(1+\curl\curl\right)$ is the graph inner
product of $\curl$ and so – indicating inner product by the respective
spaces – we have for all $\phi,\psi\in N\left(1+\curl\curl\right)$,
i.e. with $\curl\curl\phi=-\phi$ and $\psi=-\curl\curl\psi$, indeed
that
\begin{eqnarray*}
\left\langle \overset{\bullet}{\curl}\phi|\psi\right\rangle _{N\left(1+\curl\curl\right)} & = & \left\langle \iota_{\curl}^{*}\curl\iota_{\curl}\phi|\psi\right\rangle _{N\left(1+\curl\curl\right)}\\
 & = & \left\langle \curl\phi|\psi\right\rangle _{D\left(\curl\right)}\\
 & \coloneqq & \left\langle \curl\phi|\psi\right\rangle _{L^{2}\left(\Omega\right)^{3}}+\left\langle \curl\curl\phi|\curl\psi\right\rangle _{L^{2}\left(\Omega\right)^{3}}\\
 & = & -\left\langle \curl\phi|\curl\curl\psi\right\rangle _{L^{2}\left(\Omega\right)^{3}}-\left\langle \phi|\curl\psi\right\rangle _{L^{2}\left(\Omega\right)^{3}}\\
 & = & -\left\langle \phi|\curl\psi\right\rangle _{D\left(\curl\right)}\\
 & = & -\left\langle \phi|\iota_{\curl}^{*}\curl\iota_{\curl}\psi\right\rangle _{N\left(1+\curl\curl\right)}\\
 & = & -\left\langle \phi|\overset{\bullet}{\curl}\psi\right\rangle _{N\left(1+\curl\curl\right)}.
\end{eqnarray*}

\subsection{\label{subsec:Inhomogeneous-Initial-Boundary}Inhomogeneous Initial
Boundary Value Problems}

With the above boundary space set-up we can for example discuss now
inhomogeneous boundary conditions in the sense that we are looking
for a solution
\begin{align}
\left(\partial_{0}M_{0}+M_{1}+\left(\begin{array}{cccc}
0 & -\Div & 0 & 0\\
-\Grad & 0 & 0 & 0\\
0 & 0 & 0 & -\curl\\
0 & 0 & \curl & 0
\end{array}\right)\right)\left(\begin{array}{c}
v\\
T\\
E\\
H
\end{array}\right)=\label{eq:inhom-evo}\\
=\left(\begin{array}{c}
F_{0}\\
G_{0}\\
-J_{0}\\
F_{1}
\end{array}\right)\nonumber 
\end{align}
with
\begin{equation}
\begin{array}{rcl}
v-\iota_{\Grad}v_{\dot{\Omega}} & \in & H_{\nu,-1}\left(\mathbb{R},D\left(\interior\Grad\right)\right)\cap H_{\nu,0}\left(\mathbb{R},L^{2}\left(\Omega,\mathbb{C}^{3}\right)\right),\\
E-\iota_{\curl}E_{\dot{\Omega}} & \in & H_{\nu,-1}\left(\mathbb{R},D\left(\interior\curl\right)\right)\cap H_{\nu,0}\left(\mathbb{R},L^{2}\left(\Omega,\mathbb{C}^{3}\right)\right),
\end{array}\label{eq:inhom-bc}
\end{equation}
where 
\begin{equation}
\begin{array}{rcl}
v_{\dot{\Omega}} & \in & H_{\nu,1}\left(\mathbb{R},N\left(1-\Div\Grad\right)\right),\\
E_{\dot{\Omega}} & \in & H_{\nu,1}\left(\mathbb{R},N\left(1+\curl\curl\right)\right),
\end{array}\label{eq:bdata}
\end{equation}
are given (generalized) boundary data. The solution theory of this
problem can be obtained from solving the evo-system
\begin{align*}
\left(\partial_{0}M_{0}+M_{1}+\left(\begin{array}{cccc}
0 & -\Div & 0 & 0\\
-\interior\Grad & 0 & 0 & 0\\
0 & 0 & 0 & -\curl\\
0 & 0 & \interior\curl & 0
\end{array}\right)\right)\left(\begin{array}{c}
v-\iota_{\Grad}v_{\dot{\Omega}}\\
T\\
E-\iota_{\curl}E_{\dot{\Omega}}\\
H
\end{array}\right)=\\
=\left(\begin{array}{c}
F_{0}\\
G_{0}\\
-J_{0}\\
F_{1}
\end{array}\right)-\left(\partial_{0}M_{0}+M_{1}\right)\left(\begin{array}{c}
\iota_{\Grad}v_{\dot{\Omega}}\\
0\\
\iota_{\curl}E_{\dot{\Omega}}\\
0
\end{array}\right)+\left(\begin{array}{c}
0\\
\iota_{\Div}\overset{\bullet}{\Grad}\:v_{\dot{\Omega}}\\
0\\
-\iota_{\curl}\overset{\bullet}{\curl}\:E_{\dot{\Omega}}
\end{array}\right),
\end{align*}
where we note that
\begin{align*}
\iota_{\Div}\overset{\bullet}{\Grad}v_{\dot{\Omega}} & =\iota_{\Div}\iota_{\Div}^{*}\Grad\iota_{\Grad}v_{\dot{\Omega}},\\
 & =\Grad\iota_{\Grad}v_{\dot{\Omega}},
\end{align*}
\begin{align*}
\iota_{\curl}\overset{\bullet}{\curl}E_{\dot{\Omega}} & =\iota_{\curl}\iota_{\curl}^{*}\curl\iota_{\curl}E_{\dot{\Omega}},\\
 & =\curl\iota_{\curl}E_{\dot{\Omega}}.
\end{align*}
\begin{rem} A similar approach can be used to implement initial conditions
by simply solving the evo-system\footnote{Here $\left(\chi_{_{\oi0\infty}}\otimes U_{0}\right)\::=\chi_{_{\oi0\infty}}\left(t\right)\;U_{0}$
for $t\in\mathbb{R}$. }
\[
\begin{array}{r}
\left(\partial_{0}M_{0}+M_{1}\left(\partial_{0}^{-1}\right)+A\right)U=\\
=\left(\begin{array}{c}
F_{0}\\
G_{0}\\
-J_{0}\\
F_{1}
\end{array}\right)-M_{1}\left(\partial_{0}^{-1}\right)\left(\chi_{_{\oi0\infty}}\otimes\left(\begin{array}{c}
v_{0}\\
T_{0}\\
E_{0}\\
H_{0}
\end{array}\right)\right)+\chi_{_{\oi0\infty}}\otimes A\left(\begin{array}{c}
v_{0}\\
T_{0}\\
E_{0}\\
H_{0}
\end{array}\right),
\end{array}
\]
where 
\[
A=\left(\begin{array}{cccc}
0 & -\Div & 0 & 0\\
-\interior\Grad & 0 & 0 & 0\\
0 & 0 & 0 & -\curl\\
0 & 0 & \interior\curl & 0
\end{array}\right)
\]
and $M_{0}\left(\begin{array}{c}
v_{0}\\
T_{0}\\
E_{0}\\
H_{0}
\end{array}\right)$ with $\left(\begin{array}{c}
v_{0}\\
T_{0}\\
E_{0}\\
H_{0}
\end{array}\right)\in D\left(A\right)$ describe the initial data. The desired solution $\left(\begin{array}{c}
v\\
T\\
E\\
H
\end{array}\right)$ can now be easily reconstructed from 
\[
\left(\begin{array}{c}
v\\
T\\
E\\
H
\end{array}\right)=U+\chi_{_{\oi0\infty}}\otimes\left(\begin{array}{c}
v_{0}\\
T_{0}\\
E_{0}\\
H_{0}
\end{array}\right).
\]
It is for this reason that we have simplified the discussion to vanishing
initial data, compare \cite[Chapter 6]{PDE_DeGruyter}.\end{rem}

\subsection{\label{subsec:Leontovich-Type-Boundary}Leontovich Type Boundary
Conditions as Dynamics on Boundary Data Spaces}

\subsubsection{Translating Particular Model Boundary Conditions}

We recall from \cite{zbMATH05929394} the two boundary conditions:

\begin{align*}
n\times H_{t}-n\times Q^{*}v+E_{t} & =0\text{ on }\partial\Omega,\\
Tn-Q\left(n\times E_{t}\right)+\left(1+\alpha\partial_{0}^{-1}\right)v & =0\text{ on }\partial\Omega,
\end{align*}
where $E_{t},\:H_{t}$ denote the tangential components of $E,\:H$,
respectively, and $Q$, $\alpha$ are certain matrix-valued functions. 

With $n\times$ replaced by $\overset{\bullet}{\curl}$, $Tn$ by
$\overset{\bullet}{\Div}\:\iota_{\Div}^{*}T$ and $E_{t},\:H_{t}$
replaced by $\iota_{\curl}^{*}E,\:\iota_{\curl}^{*}H$, we get

\begin{equation}
\begin{array}{rcl}
\overset{\bullet}{\curl}\:\iota_{\curl}^{*}H-\overset{\bullet}{\curl}Q^{*}\iota_{\Grad}^{*}v+\iota_{\curl}^{*}E & = & 0\\
\overset{\bullet}{\Div}\:\iota_{\Div}^{*}T-Q\;\overset{\bullet}{\curl}\iota_{\curl}^{*}E+\left(1+\alpha\partial_{0}^{-1}\right)\iota_{\Grad}^{*}v & = & 0
\end{array}\label{eq:bc}
\end{equation}
In this now
\begin{eqnarray*}
Q:N\left(1+\curl\curl\right) & \to & N\left(1-\Div\Grad\right)\\
\alpha:N\left(1-\Div\Grad\right) & \to & N\left(1-\Div\Grad\right)
\end{eqnarray*}
are boundary operators. This translation yields boundary conditions
\eqref{eq:bc} which are in a form that allows again generalization
to arbitrary non-empty open sets for $\Omega$, which is one of our
main goals here. 

To motivate this translation process we note that for all $\Phi\in D\left(\curl\right)$
\begin{align}
 & \left\langle \iota_{\curl}^{*}\Phi|\overset{\bullet}{\curl}\:\iota_{\curl}^{*}H\right\rangle _{N\left(1+\curl\curl\right)}=\nonumber \\
 & =\left\langle \iota_{\curl}\iota_{\curl}^{*}\Phi|\curl\:\iota_{\curl}\iota_{\curl}^{*}H\right\rangle _{D\left(\curl\right)}\nonumber \\
 & =\left\langle \iota_{\curl}\iota_{\curl}^{*}\Phi|\curl\:\iota_{\curl}\iota_{\curl}^{*}H\right\rangle _{0}+\nonumber \\
 & +\left\langle \curl\:\iota_{\curl}\iota_{\curl}^{*}\Phi|\curl\curl\:\iota_{\curl}\iota_{\curl}^{*}H\right\rangle _{0}\label{eq:curl-BT}\\
 & =\left\langle \iota_{\curl}\iota_{\curl}^{*}\Phi|\curl\:\iota_{\curl}\iota_{\curl}^{*}H\right\rangle _{0}-\left\langle \curl\:\iota_{\curl}\iota_{\curl}^{*}\Phi|\:\iota_{\curl}\iota_{\curl}^{*}H\right\rangle _{0}\nonumber \\
 & =\left\langle \Phi|\curl\:H\right\rangle _{0}-\left\langle \curl\Phi|\:H\right\rangle _{0}\nonumber \\
 & =\left\langle \Phi|n\times H\right\rangle _{L^{2}\left(\partial\Omega\right)}\nonumber \\
 & =\left\langle \left(\gamma_{-n\times n\times}\iota_{\curl}\right)\iota_{\curl}^{*}\Phi|\left(\gamma_{n\times}\iota_{\curl}\right)\iota_{\curl}^{*}H\right\rangle _{L^{2}\left(\partial\Omega\right)}\nonumber \\
 & =\left\langle \left(\gamma_{-n\times n\times}\iota_{\curl}\right)\iota_{\curl}^{*}\Phi|R_{X}\left(\gamma_{n\times}\iota_{\curl}\right)\iota_{\curl}^{*}H\right\rangle _{X}\nonumber \\
 & =\left\langle \iota_{\curl}^{*}\Phi|\left(\gamma_{-n\times n\times}\iota_{\curl}\right)^{*}R_{X}\left(\gamma_{n\times}\iota_{\curl}\right)\iota_{\curl}^{*}H\right\rangle _{N\left(1+\curl\curl\right)}\nonumber 
\end{align}
and so
\begin{eqnarray*}
\overset{\bullet}{\curl}\:\iota_{\curl}^{*}H & = & \left(\gamma_{-n\times n\times}\iota_{\curl}\right)^{*}R_{X}\left(\gamma_{n\times}\iota_{\curl}\right)\iota_{\curl}^{*}H\\
R_{X}^{*}\left(\left(\gamma_{-n\times n\times}\iota_{\curl}\right)^{-1}\right)^{*}\overset{\bullet}{\curl}\:\iota_{\curl}^{*}H & = & \left(\gamma_{n\times}\iota_{\curl}\right)\iota_{\curl}^{*}H\\
 & = & \gamma_{n\times}H
\end{eqnarray*}

Here $R_{X}:Y\to X$ denotes the associated Riesz mapping and 
\begin{eqnarray*}
\gamma_{-n\times n\times}:D\left(\curl\right) & \to & X\\
\gamma_{n\times}:D\left(\curl\right) & \to & Y
\end{eqnarray*}
are suitable continuous boundary trace surjections with $X,Y$ being
$L^{2}\left(\partial\Omega\right)$-dual Hilbert spaces (we avoid
the intricate details here, see e.g. \cite{zbMATH01866780}, for more
specifics) and
\[
N\left(\gamma_{-n\times n\times}\right)=N\left(\gamma_{n\times}\right)=D\left(\interior\curl\right).
\]
Then
\begin{eqnarray*}
\gamma_{-n\times n\times}\:\iota_{\curl}:N\left(1+\curl\curl\right) & \to & X\\
\gamma_{n\times}\:\iota_{\curl}:N\left(1+\curl\curl\right) & \to & Y
\end{eqnarray*}
are continuous bijections.

\noindent Similarly, for all $\Phi\in D\left(\Grad\right)$
\begin{align}
 & \left\langle \iota_{\Grad}^{*}\Phi|\overset{\bullet}{\Div}\:\iota_{\Div}^{*}T\right\rangle _{N\left(1-\Div\Grad\right)}=\nonumber \\
 & =\left\langle \iota_{\Grad}\iota_{\Grad}^{*}\Phi|\Div\:\iota_{\Div}\iota_{\Div}^{*}T\right\rangle _{D\left(\Grad\right)}\nonumber \\
 & =\left\langle \iota_{\Grad}\iota_{\Grad}^{*}\Phi|\Div\:\iota_{\Div}\iota_{\Div}^{*}T\right\rangle _{0}\nonumber \\
 & +\left\langle \Grad\iota_{\Grad}\iota_{\Grad}^{*}\Phi|\Grad\Div\:\iota_{\Div}\iota_{\Div}^{*}T\right\rangle _{0}\label{eq:Grad-Div-BT}\\
 & =\left\langle \Phi|\Div\:\iota_{\Div}\iota_{\Div}^{*}T\right\rangle _{0}+\left\langle \Grad\Phi|\:\iota_{\Div}\iota_{\Div}^{*}T\right\rangle _{0}\nonumber \\
 & =\left\langle \Phi|\Div\:T\right\rangle _{0}+\left\langle \Grad\Phi|\:T\right\rangle _{0}\nonumber \\
 & =\left\langle \Phi|Tn\right\rangle _{L^{2}\left(\partial\Omega\right)}\nonumber \\
 & =\left\langle \left(\gamma_{1}\iota_{\Grad}\right)\iota_{\Grad}^{*}\Phi|\left(\gamma_{\,\cdot\:n}\iota_{\Div}\right)\iota_{\Div}^{*}T\right\rangle _{L^{2}\left(\partial\Omega\right)}\nonumber \\
 & =\left\langle \left(\gamma_{1}\iota_{\Grad}\right)\iota_{\Grad}^{*}\Phi|R_{\tilde{X}}\left(\gamma_{\,\cdot\:n}\iota_{\Div}\right)\iota_{\Div}^{*}T\right\rangle _{\tilde{X}}\nonumber \\
 & =\left\langle \iota_{\Grad}^{*}\Phi|\left(\gamma_{1}\iota_{\Grad}\right)^{*}R_{\tilde{X}}\left(\gamma_{\,\cdot\:n}\iota_{\Div}\right)\iota_{\Div}^{*}T\right\rangle _{N\left(1-\Div\Grad\right)}\nonumber 
\end{align}
and so
\begin{eqnarray*}
\overset{\bullet}{\Div} & = & \left(\gamma_{1}\iota_{\Grad}\right)^{*}R_{\tilde{X}}\left(\gamma_{\,\cdot\:n}\iota_{\Div}\right)\\
R_{\tilde{X}}^{*}\left(\left(\gamma_{1}\iota_{\Grad}\right)^{-1}\right)^{*}\overset{\bullet}{\Div}\iota_{\Div}^{*}T & = & \left(\gamma_{\,\cdot\:n}\iota_{\Div}\right)\iota_{\Div}^{*}T\\
 & = & \gamma_{\,\cdot\:n}T
\end{eqnarray*}
Here $R_{\tilde{X}}:\tilde{Y}\to\tilde{X}$ denotes the corresponding
associated Riesz mapping and
\begin{eqnarray*}
\gamma_{1}:D\left(\Grad\right) & \to & \tilde{X}\\
\gamma_{\cdot\:n}:D\left(\Div\right) & \to & \tilde{Y}
\end{eqnarray*}
are suitable continuous boundary trace surjections with $\tilde{X},\tilde{Y}$
being $L^{2}\left(\partial\Omega\right)$-dual Hilbert spaces and
\[
N\left(\gamma_{1}\right)=D\left(\interior\Grad\right),\:N\left(\gamma_{\cdot\:n}\right)=D\left(\interior\Div\right).
\]
Then
\begin{eqnarray*}
\gamma_{1}\iota_{\Grad}:N\left(1-\Div\Grad\right) & \to & \tilde{X}\\
\gamma_{\cdot\:n}\iota_{\Div}:N\left(1-\Grad\Div\right) & \to & \tilde{Y}
\end{eqnarray*}
are continuous bijections.

Both instances are showing a close, formal connection, which we have
taken as a justification for the proposed generalization for boundary
terms.

~

\subsubsection{An Evo-System Set-Up}

We shall, however, implement the boundary constraints \eqref{eq:bc}
not as typical boundary conditions but by appending, in the spirit
of abstract $\grad-\dive$~-~systems, see \cite{Picard20164888},
the differential equations in $\Omega$ by dynamical equations on
the boundary spaces. Hence, we consider a system of the form
\begin{align*}
\left(\partial_{0}M_{0}+M_{1}\left(\partial_{0}^{-1}\right)+A\right)\left(\begin{array}{c}
v\\
\left(\begin{array}{c}
T\\
\tau_{T}
\end{array}\right)\\
E\\
\left(\begin{array}{c}
H\\
\tau_{H}
\end{array}\right)
\end{array}\right) & =\left(\begin{array}{c}
F_{0}\\
\left(\begin{array}{c}
G_{0}\\
g_{0}
\end{array}\right)\\
-J_{0}\\
\left(\begin{array}{c}
F_{1}\\
f_{1}
\end{array}\right)
\end{array}\right),
\end{align*}
where 
\[
A=\left(\begin{array}{cccc}
0 & -\left(\begin{array}{c}
-\Grad\\
\iota_{\Grad}^{*}
\end{array}\right)^{*} & 0 & \quad\left(\begin{array}{cc}
0 & 0\end{array}\right)\\
\left(\begin{array}{c}
-\Grad\\
\iota_{\Grad}^{*}
\end{array}\right) & \left(\begin{array}{cc}
0 & 0\\
0 & 0
\end{array}\right) & \left(\begin{array}{c}
0\\
0
\end{array}\right) & \quad\left(\begin{array}{cc}
0 & 0\end{array}\right)\\
0 & \left(\begin{array}{cc}
0 & 0\end{array}\right) & 0 & \quad-\left(\begin{array}{c}
\curl\\
\iota_{\curl}^{*}
\end{array}\right)^{*}\\
\left(\begin{array}{c}
0\\
0
\end{array}\right) & \left(\begin{array}{c}
0\\
0
\end{array}\right) & \left(\begin{array}{c}
\curl\\
\iota_{\curl}^{*}
\end{array}\right) & \quad\left(\begin{array}{cc}
0 & 0\\
0 & 0
\end{array}\right)
\end{array}\right)
\]
is by construction – as desired – skew-selfadjoint and $M_{0}$, $M_{1}\left(\partial_{0}^{-1}\right)$
are to be specified later.

To analyze the operator $A$ closer we need to obtain a better understanding
of $\left(\begin{array}{c}
-\Grad\\
\iota_{\Grad}^{*}
\end{array}\right)^{*}$ and $\left(\begin{array}{c}
-\curl\\
\iota_{\curl}^{*}
\end{array}\right)^{*}$. We first observe that
\[
\left(\begin{array}{c}
-\interior\Grad\\
0
\end{array}\right)\subseteq\left(\begin{array}{c}
-\Grad\\
\iota_{\Grad}^{*}
\end{array}\right),\,\left(\begin{array}{c}
\interior\curl\\
0
\end{array}\right)\subseteq\left(\begin{array}{c}
\curl\\
\iota_{\curl}^{*}
\end{array}\right),
\]
which implies that
\begin{align*}
\left(\begin{array}{c}
-\Grad\\
\iota_{\Grad}^{*}
\end{array}\right)^{*} & \subseteq\left(\begin{array}{c}
-\interior\Grad\\
0
\end{array}\right)^{*}=\left(\begin{array}{cc}
\Div & \:0\end{array}\right),\\
\left(\begin{array}{c}
\curl\\
\iota_{\curl}^{*}
\end{array}\right)^{*} & \subseteq\,\left(\begin{array}{c}
\interior\curl\\
0
\end{array}\right)^{*}=\left(\begin{array}{cc}
\curl & \:0\end{array}\right).
\end{align*}
Thus, for all $\Phi\in D\left(\Grad\right)$ and $\left(\begin{array}{c}
T\\
\tau_{T}
\end{array}\right)\in D\left(\left(\begin{array}{c}
-\Grad\\
\iota_{\Grad}^{*}
\end{array}\right)^{*}\right)$
\begin{align*}
\left\langle -\Grad\Phi\Big|T\right\rangle +\left\langle \iota_{\Grad}^{*}\Phi\Big|\tau_{T}\right\rangle  & =\left\langle \left(\begin{array}{c}
-\Grad\\
\iota_{\Grad}^{*}
\end{array}\right)\Phi\Big|\left(\begin{array}{c}
T\\
\tau_{T}
\end{array}\right)\right\rangle \\
 & =\left\langle \Phi\Big|\left(\begin{array}{c}
-\Grad\\
\iota_{\Grad}^{*}
\end{array}\right)^{*}\left(\begin{array}{c}
T\\
\tau_{T}
\end{array}\right)\right\rangle \\
 & =\left\langle \Phi\Big|\Div T\right\rangle .
\end{align*}
Since 
\[
\left\langle \Grad\Phi\Big|T\right\rangle +\left\langle \Phi\Big|\Div T\right\rangle =0
\]
for $\Phi\in D\left(\interior\Grad\right)$ or $T\in D\left(\interior\Div\right)$
we have for $\Phi\in D\left(\Grad\right)$ and $T\in D\left(\Div\right)$
\begin{eqnarray*}
 &  & \left\langle \Grad\Phi\Big|T\right\rangle +\left\langle \Phi\Big|\Div T\right\rangle =\\
 &  & =\left\langle \Grad\iota_{\Grad}\iota_{\Grad}^{*}\Phi\Big|\iota_{\Div}\iota_{\Div}^{*}T\right\rangle +\left\langle \iota_{\Grad}\iota_{\Grad}^{*}\Phi\Big|\Div\iota_{\Div}\iota_{\Div}^{*}T\right\rangle +\\
 &  & +\left\langle \Grad\iota_{\Grad}\iota_{\Grad}^{*}\Phi\Big|\left(1-\iota_{\Div}\iota_{\Div}^{*}\right)T\right\rangle +\\
 &  & +\left\langle \left(1-\iota_{\Grad}\iota_{\Grad}^{*}\right)\Phi\Big|\Div\iota_{\Div}\iota_{\Div}^{*}T\right\rangle +\\
 &  & +\left\langle \iota_{\Grad}\iota_{\Grad}^{*}\Phi\Big|\Div\left(1-\iota_{\Div}\iota_{\Div}^{*}\right)T\right\rangle +\\
 &  & +\left\langle \Grad\left(1-\iota_{\Grad}\iota_{\Grad}^{*}\right)\Phi\Big|\iota_{\Div}\iota_{\Div}^{*}T\right\rangle +\\
 &  & +\left\langle \left(1-\iota_{\Grad}\iota_{\Grad}^{*}\right)\Phi\Big|\Div\left(1-\iota_{\Div}\iota_{\Div}^{*}\right)T\right\rangle +\\
 &  & +\left\langle \Grad\left(1-\iota_{\Grad}\iota_{\Grad}^{*}\right)\Phi\Big|\left(1-\iota_{\Div}\iota_{\Div}^{*}\right)T\right\rangle \\
 &  & =\left\langle \Grad\iota_{\Grad}\iota_{\Grad}^{*}\Phi\Big|\iota_{\Div}\iota_{\Div}^{*}T\right\rangle +\left\langle \iota_{\Grad}\iota_{\Grad}^{*}\Phi\Big|\Div\iota_{\Div}\iota_{\Div}^{*}T\right\rangle 
\end{eqnarray*}
and recalling \eqref{eq:Grad-Div-BT} we calculate with this for all
$\Phi\in D\left(\Grad\right)$ and $\left(\begin{array}{c}
T\\
\tau_{T}
\end{array}\right)\in D\left(\left(\begin{array}{c}
-\Grad\\
\iota_{\Grad}^{*}
\end{array}\right)^{*}\right)\subseteq D\left(\Div\right)$
\begin{align*}
 & \left\langle \iota_{\Grad}^{*}\Phi\Big|\tau_{T}\right\rangle _{N\left(1-\Div\Grad\right)}=\\
 & =\left\langle \Grad\iota_{\Grad}\iota_{\Grad}^{*}\Phi\Big|\iota_{\Div}\iota_{\Div}^{*}T\right\rangle +\left\langle \iota_{\Grad}\iota_{\Grad}^{*}\Phi\Big|\Div\iota_{\Div}\iota_{\Div}^{*}T\right\rangle ,\\
 & =\frac{1}{2}\left\langle \Grad\iota_{\Grad}\iota_{\Grad}^{*}\Phi\Big|\iota_{\Div}\iota_{\Div}^{*}T\right\rangle +\frac{1}{2}\left\langle \iota_{\Grad}\iota_{\Grad}^{*}\Phi\Big|\Div\iota_{\Div}\iota_{\Div}^{*}T\right\rangle +\\
 & +\frac{1}{2}\left\langle \Grad\iota_{\Grad}\iota_{\Grad}^{*}\Phi\Big|\Grad\Div\iota_{\Div}\iota_{\Div}^{*}T\right\rangle +\\
 & +\frac{1}{2}\left\langle \Div\Grad\iota_{\Grad}\iota_{\Grad}^{*}\Phi\Big|\Div\iota_{\Div}\iota_{\Div}^{*}T\right\rangle ,\\
 & =\frac{1}{2}\left\langle \Grad\:\iota_{\Grad}\iota_{\Grad}^{*}\Phi\Big|\iota_{\Div}\iota_{\Div}^{*}T\right\rangle _{D\left(\Div\right)}+\\
 & +\frac{1}{2}\left\langle \iota_{\Grad}\iota_{\Grad}^{*}\Phi\Big|\Div\:\iota_{\Div}\iota_{\Div}^{*}T\right\rangle _{D\left(\Grad\right)},\\
 & =\frac{1}{2}\left\langle \iota_{\Div}^{*}\Grad\:\iota_{\Grad}\iota_{\Grad}^{*}\Phi\Big|\iota_{\Div}^{*}T\right\rangle _{N\left(1-\Grad\Div\right)}+\\
 & +\frac{1}{2}\left\langle \iota_{\Grad}^{*}\Phi\Big|\iota_{\Grad}^{*}\Div\:\iota_{\Div}\iota_{\Div}^{*}T\right\rangle _{N\left(1-\Div\Grad\right)}\\
 & =\frac{1}{2}\left\langle \overset{\bullet}{\Grad}\iota_{\Grad}^{*}\Phi\Big|\iota_{\Div}^{*}T\right\rangle _{N\left(1-\Grad\Div\right)}+\frac{1}{2}\left\langle \iota_{\Grad}^{*}\Phi\Big|\overset{\bullet}{\Div}\iota_{\Div}^{*}T\right\rangle _{N\left(1-\Div\Grad\right)},\\
 & =\left\langle \iota_{\Grad}^{*}\Phi\Big|\overset{\bullet}{\Div}\:\iota_{\Div}^{*}T\right\rangle _{N\left(1-\Div\Grad\right)}
\end{align*}
and so
\[
\tau_{T}=\overset{\bullet}{\Div}\:\iota_{\Div}^{*}T.
\]
Similarly, we find for all $\Phi\in D\left(\curl\right)$ and $\left(\begin{array}{c}
H\\
\tau_{H}
\end{array}\right)\in D\left(\left(\begin{array}{c}
\curl\\
\iota_{\curl}^{*}
\end{array}\right)^{*}\right)\subseteq D\left(\curl\right)$
\begin{align*}
\left\langle \curl\Phi\Big|E\right\rangle +\left\langle \iota_{\curl}^{*}\Phi\Big|\tau_{H}\right\rangle  & _{N\left(1+\curl\curl\right)}=\left\langle \left(\begin{array}{c}
\curl\\
\iota_{\curl}^{*}
\end{array}\right)\Phi\Big|\left(\begin{array}{c}
H\\
\tau_{H}
\end{array}\right)\right\rangle \\
 & =\left\langle \Phi\Big|\left(\begin{array}{c}
\curl\\
\iota_{\curl}^{*}
\end{array}\right)^{*}\left(\begin{array}{c}
H\\
\tau_{H}
\end{array}\right)\right\rangle \\
 & =\left\langle \Phi\Big|\curl H\right\rangle 
\end{align*}
leading with \eqref{eq:curl-BT} to

\begin{align*}
 & \left\langle \iota_{\curl}^{*}\Phi\Big|\tau_{H}\right\rangle _{N\left(1+\curl\curl\right)}=\\
 & =-\left\langle \curl\iota_{\curl}\iota_{\curl}^{*}\Phi\Big|\iota_{\curl}\iota_{\curl}^{*}E\right\rangle +\left\langle \iota_{\curl}\iota_{\curl}^{*}\Phi\Big|\curl\iota_{\curl}\iota_{\curl}^{*}E\right\rangle \\
 & =-\frac{1}{2}\left\langle \curl\iota_{\curl}\iota_{\curl}^{*}\Phi\Big|\iota_{\curl}\iota_{\curl}^{*}E\right\rangle +\\
 & +\frac{1}{2}\left\langle \iota_{\curl}\iota_{\curl}^{*}\Phi\Big|\curl\iota_{\curl}\iota_{\curl}^{*}E\right\rangle +\\
 & +\frac{1}{2}\left\langle \curl\iota_{\curl}\iota_{\curl}^{*}\Phi\Big|\curl\curl\iota_{\curl}\iota_{\curl}^{*}E\right\rangle +\\
 & -\frac{1}{2}\left\langle \curl\curl\iota_{\curl}\iota_{\curl}^{*}\Phi\Big|\curl\iota_{\curl}\iota_{\curl}^{*}E\right\rangle \\
 & =-\frac{1}{2}\left\langle \curl\iota_{\curl}\iota_{\curl}^{*}\Phi\Big|\iota_{\curl}\iota_{\curl}^{*}E\right\rangle _{D\left(\curl\right)}+\\
 & +\frac{1}{2}\left\langle \iota_{\curl}\iota_{\curl}^{*}\Phi\Big|\curl\iota_{\curl}\iota_{\curl}^{*}E\right\rangle _{D\left(\curl\right)}\\
 & =-\frac{1}{2}\left\langle \overset{\bullet}{\curl}\:\iota_{\curl}^{*}\Phi\Big|\iota_{\curl}^{*}E\right\rangle _{N\left(1-\curl\curl\right)}+\\
 & +\frac{1}{2}\left\langle \iota_{\curl}^{*}\Phi\Big|\overset{\bullet}{\curl}\:\iota_{\curl}^{*}E\right\rangle _{N\left(1-\curl\curl\right)}\\
 & =\left\langle \iota_{\curl}^{*}\Phi\Big|\overset{\bullet}{\curl}\:\iota_{\curl}^{*}E\right\rangle _{N\left(1-\curl\curl\right)}
\end{align*}
and so 
\[
\tau_{H}=\overset{\bullet}{\curl}\:\iota_{\curl}^{*}H.
\]
With this the boundary constraints take the form

\begin{align*}
\tau_{H}-\overset{\bullet}{\curl}\:Q^{*}\iota_{\Grad}^{*}v+\iota_{\curl}^{*}E & =0,\\
\tau_{T}-Q\;\overset{\bullet}{\curl}\iota_{\curl}^{*}E+\left(1+\alpha\partial_{0}^{-1}\right)\iota_{\Grad}^{*}v & =0,
\end{align*}
or

\[
\left(\begin{array}{c}
\tau_{H}\\
\tau_{T}
\end{array}\right)+\left(\begin{array}{cc}
1 & -\overset{\bullet}{\curl}\;Q^{*}\\
-Q\;\overset{\bullet}{\curl} & \left(1+\alpha\partial_{0}^{-1}\right)
\end{array}\right)\left(\begin{array}{c}
\iota_{\curl}^{*}E\\
\iota_{\Grad}^{*}v
\end{array}\right)=0.
\]
We calculate

\begin{align*}
 & \left(\begin{array}{cc}
1 & -\overset{\bullet}{\curl}Q^{*}\\
-Q\overset{\bullet}{\curl} & \left(1+\alpha\partial_{0}^{-1}\right)
\end{array}\right)^{-1}=\\
 & =\left(\begin{array}{cc}
1+\overset{\bullet}{\curl}Q^{*}\left(1+QQ^{*}+\alpha\partial_{0}^{-1}\right)^{-1}Q\overset{\bullet}{\curl} & \overset{\bullet}{\curl}Q^{*}\left(1+QQ^{*}+\alpha\partial_{0}^{-1}\right)^{-1}\\
\left(1+QQ^{*}+\alpha\partial_{0}^{-1}\right)^{-1}Q\overset{\bullet}{\curl} & \left(1+QQ^{*}+\alpha\partial_{0}^{-1}\right)^{-1}
\end{array}\right)
\end{align*}
and thus obtain equivalently
\begin{equation}
\begin{array}{l}
S\left(\partial_{0}^{-1}\right)\left(\begin{array}{c}
\tau_{H}\\
\tau_{T}
\end{array}\right)+\left(\begin{array}{c}
\iota_{\curl}^{*}E\\
\iota_{\Grad}^{*}v
\end{array}\right)=0.\end{array}\label{eq:beq}
\end{equation}
with $S\left(\partial_{0}^{-1}\right)$ given by
\[
\left(\begin{array}{cc}
1+\overset{\bullet}{\curl}Q^{*}\left(1+QQ^{*}+\alpha\partial_{0}^{-1}\right)^{-1}Q\overset{\bullet}{\curl} & \overset{\bullet}{\curl}Q^{*}\left(1+QQ^{*}+\alpha\partial_{0}^{-1}\right)^{-1}\\
\left(1+QQ^{*}+\alpha\partial_{0}^{-1}\right)^{-1}Q\overset{\bullet}{\curl} & \left(1+QQ^{*}+\alpha\partial_{0}^{-1}\right)^{-1}
\end{array}\right).
\]
We are now ready to formulate the material law operators
\[
M_{0}=\left(\begin{array}{cccc}
\rho_{*} & \left(\begin{array}{cc}
0 & 0\end{array}\right) & 0 & \quad\left(\begin{array}{cc}
0 & 0\end{array}\right)\\
\left(\begin{array}{c}
0\\
0
\end{array}\right) & \left(\begin{array}{cc}
C^{-1} & 0\\
0 & 0
\end{array}\right) & \left(\begin{array}{c}
0\\
0
\end{array}\right) & \quad\left(\begin{array}{cc}
C^{-1}e & 0\\
0 & 0
\end{array}\right)\\
0 & \left(\begin{array}{cc}
0 & 0\end{array}\right) & \epsilon+e^{*}C^{-1}e & \quad\left(\begin{array}{cc}
0 & 0\end{array}\right)\\
\left(\begin{array}{c}
0\\
0
\end{array}\right) & \left(\begin{array}{cc}
e^{*}C^{-1} & 0\\
0 & 0
\end{array}\right) & \left(\begin{array}{c}
0\\
0
\end{array}\right) & \quad\left(\begin{array}{cc}
\mu & 0\\
0 & 0
\end{array}\right)
\end{array}\right)
\]
and 
\begin{align*}
M_{1}\left(\partial_{0}^{-1}\right) & =\left(\begin{array}{cccc}
0 & \left(\begin{array}{cc}
0 & 0\end{array}\right) & 0 & \left(\begin{array}{cc}
0 & 0\end{array}\right)\\
\left(\begin{array}{c}
0\\
0
\end{array}\right) & M_{1,22}\left(\partial_{0}^{-1}\right) & \left(\begin{array}{c}
0\\
0
\end{array}\right) & -M_{1,24}\left(\partial_{0}^{-1}\right)\\
0 & \left(\begin{array}{cc}
0 & 0\end{array}\right) & \sigma & \left(\begin{array}{cc}
0 & 0\end{array}\right)\\
\left(\begin{array}{c}
0\\
0
\end{array}\right) & M_{1,42}\left(\partial_{0}^{-1}\right) & \left(\begin{array}{c}
0\\
0
\end{array}\right) & M_{1,44}\left(\partial_{0}^{-1}\right)
\end{array}\right)
\end{align*}
with 
\[
M_{1,44}\left(\partial_{0}^{-1}\right)=\left(\begin{array}{cc}
0 & 0\\
0 & 1+\overset{\bullet}{\curl}Q^{*}\left(1+QQ^{*}+\alpha\partial_{0}^{-1}\right)^{-1}Q\overset{\bullet}{\curl}
\end{array}\right)
\]
\begin{eqnarray*}
M_{1,42}\left(\partial_{0}^{-1}\right) & = & \left(\begin{array}{cc}
0 & 0\\
0 & \left(1+QQ^{*}+\alpha\partial_{0}^{-1}\right)^{-1}Q\;\overset{\bullet}{\curl}
\end{array}\right)\\
M_{1,24}\left(\partial_{0}^{-1}\right) & = & \left(\begin{array}{cc}
0 & \overset{\bullet}{\curl}Q^{*}\left(1+QQ^{*}+\alpha\partial_{0}^{-1}\right)^{-1}\\
0 & 0
\end{array}\right)
\end{eqnarray*}

\[
M_{1,22}\left(\partial_{0}^{-1}\right)=\left(\begin{array}{cc}
0 & 0\\
0 & \left(1+QQ^{*}+\alpha\partial_{0}^{-1}\right)^{-1}
\end{array}\right).
\]

\begin{thm}Assume that $\rho_{\ast},\varepsilon,\mu,C$ are selfadjoint
and non-negative, $Q:N\left(1+\curl\curl\right)\to N\left(1-\Div\Grad\right)$.
Furthermore, we assume $\rho_{\ast},\mu,C\gg0$ and $\nu\epsilon+\Re\sigma\gg0$
uniformly for all sufficiently large $\nu\in\oi0\infty$. Then, $M_{0}$
and $M_{1}\left(\partial_{0}^{-1}\right)$ satisfy the condition (\ref{eq:pos-def11})
and hence, the corresponding problem of piezo-electricity with dynamics
on the boundary data space is also a well-posed evo-system. \end{thm}

\begin{proof}Obviously, $M_{0}$ is selfadjoint. Moreover, since

\[
\nu\left(\begin{array}{cc}
C^{-1} & C^{-1}e\\
e^{\ast}C^{-1} & \varepsilon+e^{\ast}C^{-1}e
\end{array}\right)+\left(\begin{array}{cc}
0 & 0\\
0 & \Re\sigma
\end{array}\right)\gg0
\]
uniformly for all sufficiently large $\nu\in\oi0\infty$~, the assertion
follows.

Indeed, noting that
\begin{align*}
\Re S\left(0\right)=\\
=\Re\left(\begin{array}{cc}
1+\overset{\bullet}{\curl}\:Q^{*}\left(1+QQ^{*}\right)^{-1}Q\overset{\bullet}{\curl} & \quad\overset{\bullet}{\curl}Q^{*}\left(1+QQ^{*}\right)^{-1}\\
\left(1+QQ^{*}\right)^{-1}Q\overset{\bullet}{\curl} & \left(1+QQ^{*}\right)^{-1}
\end{array}\right),\\
=\left(\begin{array}{cc}
\overset{\bullet}{1+\curl}Q^{*}\left(1+QQ^{*}\right)^{-1}Q\overset{\bullet}{\curl} & 0\\
0 & \left(1+QQ^{*}\right)^{-1}
\end{array}\right)\geq1
\end{align*}
the desired result follows from the general result of Theorem \ref{thm:well-posed}.\end{proof}

\begin{rem}~
\begin{enumerate}
\item For simplicity we have assumed that there is no thermal interaction.
There is, however, no major obstacle to incorporate such interaction
along the lines of \cite{MMA:MMA3866}. Similarly, more complex boundary
constraints of abstract $\grad-\dive$-type could be implemented following
the lead of the present framework. 
\item Although we have merely generalized a known model system, it is clear
from the set-up that more complicated situations are easily incorporated.
For example
\begin{enumerate}
\item apart from the generalized coefficients we can of course allow inhomogeneous
data with no extra provision, since the ``boundary conditions''
are built into the system as part of the evo-system,
\item the material laws can be even more general as long as requirement
\eqref{eq:pos-def11} remains satisfied.
\end{enumerate}
\item As stated in Remark \ref{rem:If-continuous,-linear}, equivalence
is a common way of obscuring the basic structure of evo-systems. In
the above we have in fact encountered such a situation.\\
If we may assume that boundary trace mappings are available, another
pertinent case is given in our present context by
\[
\mathcal{W}\left(\partial_{0}M_{0}+M_{1}\left(\partial_{0}^{-1}\right)+A\right)\mathcal{V}\left(\mathcal{V}^{-1}U\right)=\mathcal{W}F
\]
with 
\begin{eqnarray*}
\mathcal{W} & = & \left(\begin{array}{cccc}
1 & \left(\begin{array}{cc}
0 & 0\end{array}\right) & 0 & \quad\left(\begin{array}{cc}
0 & 0\end{array}\right)\\
\left(\begin{array}{c}
0\\
0
\end{array}\right) & \left(\begin{array}{cc}
1 & 0\\
0 & \gamma_{1}\iota_{\Grad}
\end{array}\right) & \left(\begin{array}{c}
0\\
0
\end{array}\right) & \quad\left(\begin{array}{cc}
0 & 0\\
0 & 0
\end{array}\right)\\
0 & \left(\begin{array}{cc}
0 & 0\end{array}\right) & 1 & \quad\left(\begin{array}{cc}
0 & 0\end{array}\right)\\
\left(\begin{array}{c}
0\\
0
\end{array}\right) & \left(\begin{array}{cc}
0 & 0\\
0 & 0
\end{array}\right) & \left(\begin{array}{c}
0\\
0
\end{array}\right) & \quad\left(\begin{array}{cc}
1 & 0\\
0 & \gamma_{-n\times n\times}\iota_{\curl}
\end{array}\right)
\end{array}\right),\\
\mathcal{V} & = & \mathcal{W}^{*}\left(\begin{array}{cccc}
1 & \left(\begin{array}{cc}
0 & 0\end{array}\right) & 0 & \quad\left(\begin{array}{cc}
0 & 0\end{array}\right)\\
\left(\begin{array}{c}
0\\
0
\end{array}\right) & \left(\begin{array}{cc}
1 & 0\\
0 & R_{\tilde{X}}
\end{array}\right) & \left(\begin{array}{c}
0\\
0
\end{array}\right) & \quad\left(\begin{array}{cc}
0 & 0\\
0 & 0
\end{array}\right)\\
0 & \left(\begin{array}{cc}
0 & 0\end{array}\right) & 1 & \quad\left(\begin{array}{cc}
0 & 0\end{array}\right)\\
\left(\begin{array}{c}
0\\
0
\end{array}\right) & \left(\begin{array}{cc}
0 & 0\\
0 & 0
\end{array}\right) & \left(\begin{array}{c}
0\\
0
\end{array}\right) & \quad\left(\begin{array}{cc}
1 & 0\\
0 & R_{X}
\end{array}\right)
\end{array}\right).
\end{eqnarray*}
The unknown is
\[
\mathcal{V}^{-1}U=\left(\begin{array}{c}
v\\
\left(\begin{array}{c}
T\\
R_{\tilde{X}}^{*}\left(\left(\gamma_{1}\iota_{\Grad}\right)^{-1}\right)^{*}\tau_{T}
\end{array}\right)\\
E\\
\left(\begin{array}{c}
H\\
R_{X}^{*}\left(\left(\gamma_{-n\times n\times}\iota_{\curl}\right)^{-1}\right)^{*}\tau_{H}
\end{array}\right)
\end{array}\right)\in H_{\nu,0}\left(\mathbb{R},\mathcal{Y}\right)
\]
with 
\[
\mathcal{Y}=L^{2}\left(\Omega,\mathbb{C}^{3}\right)\oplus\left(L^{2}\left(\Omega,\mathrm{sym}\left[\mathbb{C}^{3\times3}\right]\right)\oplus\tilde{Y}\right)\oplus L^{2}\left(\Omega,\mathbb{C}^{3}\right)\oplus\left(L^{2}\left(\Omega,\left[\mathbb{C}^{3}\right]\right)\oplus Y\right)
\]
and $\mathcal{W}F\in H_{\nu,0}\left(\mathbb{R},\mathcal{X}\right)$
with 
\[
\mathcal{X}=L^{2}\left(\Omega,\mathbb{C}^{3}\right)\oplus\left(L^{2}\left(\Omega,\mathrm{sym}\left[\mathbb{C}^{3\times3}\right]\right)\oplus\tilde{X}\right)\oplus L^{2}\left(\Omega,\mathbb{C}^{3}\right)\oplus\left(L^{2}\left(\Omega,\left[\mathbb{C}^{3}\right]\right)\oplus X\right).
\]
This is now the corresponding situation utilizing classical boundary
trace spaces. To obtain a structure preserving congruence we could
instead replace $\mathcal{V}$ by $\mathcal{W}^{*}$ in which case
\[
\left(\mathcal{W}^{-1}\right)^{*}U=\left(\begin{array}{c}
v\\
\left(\begin{array}{c}
T\\
\left(\left(\gamma_{1}\iota_{\Grad}\right)^{-1}\right)^{*}\tau_{T}
\end{array}\right)\\
E\\
\left(\begin{array}{c}
H\\
\left(\left(\gamma_{-n\times n\times}\iota_{\curl}\right)^{-1}\right)^{*}\tau_{H}
\end{array}\right)
\end{array}\right)\in H_{\nu,0}\left(\mathbb{R},\mathcal{X}\right)
\]
is now the new unknown.
\end{enumerate}
\end{rem}

~

\section{Summary}

We have generalized a piezo-electromagnetism model with Dirichlet
type boundary conditions to arbitrary non-empty open sets, as well
as to include operator coefficients, indeed to general material laws.
The resulting evo-system in a non-empty open set $\Omega$ and on
boundary data spaces, which includes inhomogeneous volume and boundary
data, has been investigated for evolutionary well-posedness, i.e.
Hadamard well-posedness and causality. Based on this the model has
been extended to include also a Leontovich type boundary coupling
via an additional set of dynamic equations on spaces characterizing
boundary data.

~

\end{document}